\documentclass[11pt]{article}
\usepackage[T1]{fontenc}
\usepackage{lmodern,amsmath,amsthm,amsfonts,amssymb,graphicx,float,wrapfig,calc,microtype,thmtools,underscore,mathtools,paralist,thm-restate} 
\usepackage[usenames,dvipsnames,svgnames,table]{xcolor}
\usepackage{breakurl}
\usepackage[unicode=true]{hyperref}
\hypersetup{ 
	colorlinks,
	linkcolor={black},
	citecolor={black},
	urlcolor={blue!60!black},
	pdftitle={Colourings of Strong Products},
	pdfauthor={Louis Esperet, David~R.~Wood}} 
\usepackage[noabbrev,capitalise]{cleveref}
\crefname{lem}{Lemma}{Lemmas}
\crefname{thm}{Theorem}{Theorems}
\crefname{cor}{Corollary}{Corollaries}
\crefname{prop}{Proposition}{Propositions}
\crefname{conj}{Conjecture}{Conjectures}
\crefname{qu}{Question}{Questions}
\crefname{openproblem}{Open Problem}{Open Problems}
\crefformat{equation}{(#2#1#3)}
\Crefformat{equation}{Equation #2(#1)#3}

\newcommand{\defn}[1]{\textcolor{Maroon}{\emph{#1}}}

\newcommand{\bigchi}{\raisebox{1.55pt}{\scalebox{1.2}{\ensuremath\chi}}}

\newcommand{\fchi}{\bigchi^f\hspace*{-0.2ex}}
\newcommand{\cchi}{\bigchi_{\star}\hspace*{-0.2ex}}

\newcommand{\dchi}{\bigchi\hspace*{-0.1ex}_{\Delta}\hspace*{-0.2ex}}

\newcommand{\cfchi}{\bigchi^f_{\star}\hspace*{-0.1ex}}
\newcommand{\dfchi}{\bigchi^f_{\Delta}\hspace*{-0.1ex}}
\newcommand{\CartProd}{\mathbin{\square}}
\renewcommand{\Pr}{\,\mathbb{P}}

\newcommand{\TT}{\mathcal{T}}
\newcommand{\GG}{\mathcal{G}}
\newcommand{\STAR}{\mathcal{S}}

\newcommand{\CC}{\mathcal{C}}

\usepackage[longnamesfirst,numbers,sort&compress]{natbib}
\makeatletter
\def\NAT@spacechar{~}
\makeatother
\usepackage[bmargin=25mm,tmargin=25mm,lmargin=35mm,rmargin=35mm]{geometry}
\allowdisplaybreaks
\newcommand{\half}{\ensuremath{\protect\tfrac{1}{2}}}

\DeclarePairedDelimiter{\ceil}{\lceil}{\rceil}
\renewcommand{\ge}{\geqslant}
\renewcommand{\le}{\leqslant}
\renewcommand{\geq}{\geqslant}
\renewcommand{\leq}{\leqslant}
\DeclareMathOperator{\dist}{dist}

\DeclareMathOperator{\tw}{tw}

\renewcommand{\thefootnote}{\fnsymbol{footnote}}
\allowdisplaybreaks
\theoremstyle{plain}
\newtheorem{thm}{Theorem}
\newtheorem{lem}[thm]{Lemma}
\newtheorem{cor}[thm]{Corollary}
\newtheorem{prop}[thm]{Proposition}
\newtheorem{obs}[thm]{Observation}
\newtheorem{qu}[thm]{Question}
\theoremstyle{definition}

\newcommand{\QQ}{\mathbb{Q}}
\newcommand{\NN}{\mathbb{N}}
\newcommand{\RR}{\mathbb{R}}

\def\longequation{$$\vcenter\bgroup\advance\hsize by -9em%
	\noindent\ignorespaces\refstepcounter{equation}}%
\makeatletter%
\def\endlongequation{\egroup\eqno(\theequation)$$\global\@ignoretrue}
\makeatother

\begin{document}
	
	\author{Louis Esperet\,\footnotemark[2] \qquad David~R.~Wood\,\footnotemark[5]}
	
	\footnotetext[2]{Laboratoire G-SCOP  (CNRS, Univ.\ Grenoble Alpes), Grenoble, France 
		(\texttt{louis.esperet@grenoble-inp.fr}). Partially supported by the French ANR Projects
		GATO (ANR-16-CE40-0009-01), GrR (ANR-18-CE40-0032), TWIN-WIDTH
		(ANR-21-CE48-0014-01) and by LabEx
		PERSYVAL-lab (ANR-11-LABX-0025).}
	
	\footnotetext[5]{School of Mathematics, Monash   University, Melbourne, Australia  (\texttt{david.wood@monash.edu}). Research supported by the Australian Research Council.}
	
	\sloppy
	
	\title{\textbf{Colouring  Strong Products}}
	
	\maketitle
	
	\begin{abstract}
		Recent results show that several important graph classes can be
		embedded as subgraphs of strong products of simpler graphs classes
		(paths, small cliques, or graphs of bounded treewidth).  This paper
		develops general techniques to bound the chromatic number (and its
		popular variants, such as fractional, clustered, or defective
		chromatic number) of the strong product of general graphs with simpler
		graphs classes, such as paths, and more generally graphs of bounded
		treewidth. We also highlight important links between the study of
		(fractional) clustered colouring of strong products and other topics,
		such as asymptotic dimension in metric geometry and topology, site
		percolation in probability theory, and the Shannon capacity in information theory.
	\end{abstract}
	
	\renewcommand{\thefootnote}{\arabic{footnote}}

	\section{Introduction}
	\label{Introduction}

	The past few years have seen a renewed interest in the structure of strong
	products of graphs. One motivation is the Planar Graph Product Structure Theorem (see \cref{SubgraphsStrongProducts}), which shows that every planar
	graph is a subgraph of the strong product of a graph with bounded
	treewidth and a path. As a consequence, results on the structure of
	planar graphs can be simply deduced from the study of the structure of strong
	products of graphs (and in particular from the study of the strong product
	of a graph with a path). This theorem was preceded by several results
	that can be also stated as product structure theorems (where the host
	graph is the strong product of paths, trees, or small complete
	graphs); see \cref{SubgraphsStrongProducts} below. Note that grids in
	finite-dimensional euclidean spaces can be
	expressed as the strong product of finitely many paths, and colouring
	properties of these grids are related to important topological or
	metric 
	properties of these  spaces; see \cref{sec:hex} and \cref{sec:asdim}. 
	
	It turns out that the study of colouring properties of strong products
	of graphs has interesting connections with various problems in combinatorics, which we highlight below. In particular, in order
	to understand the chromatic number (or the clustered or defective
	chromatic number) of strong products, it is very helpful to understand the
	fractional versions of such  colourings, which have strong ties with site 
	percolation in probability theory (see \cref{sec:perco}) and Shannon capacity in information theory (see \cref{sec:shannon}).
	
	We start with the definitions of various graph products, as well as various graph colouring 
	notions that are studied in this paper. We then give an overview of our main results in \cref{sec:results}.
	
	\subsection{Definitions}\label{sec:def}

	The \defn{cartesian product} of graphs $A$ and $B$, denoted by $A\CartProd B$, is the graph with vertex set $V(A)\times V(B)$, where distinct vertices $(v,x),(w,y)\in V(A)\times V(B)$ are adjacent if: 
	$v=w$ and $xy\in E(B)$, or $x=y$ and $vw\in E(A)$.
	The \defn{direct product} of graphs $A$ and $B$, denoted by $A\times B$, is the graph with vertex set $V(A)\times V(B)$, where distinct vertices $(v,x),(w,y)\in V(A)\times V(B)$ are adjacent if $vw\in E(A)$ and $xy\in E(B)$. 
	The \defn{strong product} of graphs $A$ and $B$, denoted by $A\boxtimes B$, is the graph $(A\CartProd B)\cup (A\times B)$. 
	For graph classes $\GG_1$ and $\GG_2$, let 
	\begin{align*}
		\GG_1\CartProd \GG_2 & := \{G_1\CartProd G_2: G_1\in \GG_1,G_2\in\GG_2\}\\
		\GG_1\times \GG_2 & := \{G_1\times G_2: G_1\in \GG_1,G_2\in\GG_2\}\\
		\GG_1\boxtimes \GG_2 & := \{G_1\boxtimes G_2: G_1\in \GG_1,G_2\in\GG_2\}.
	\end{align*}
	
	A \defn{colouring} of a graph $G$ is simply a function $f:V(G)\to\mathcal{C}$ for some set $\mathcal{C}$ whose elements are called \defn{colours}. If $|\mathcal{C}| \leq k$ then $f$ is a \defn{$k$-colouring}. An edge $vw$ of $G$ is \defn{$f$-monochromatic} if $f(v)=f(w)$. An \defn{$f$-monochromatic component}, sometimes called a \defn{monochromatic component}, is a connected component of the subgraph of $G$ induced by $\{v\in V(G):f(v)=\alpha\}$ for some  $\alpha\in \mathcal{C}$. We say $f$ has \defn{clustering} $c$ if every $f$-monochromatic component has at most $c$ vertices. The \defn{$f$-monochromatic degree} of a vertex $v$ is the degree of $v$ in the monochromatic component containing $v$. Then $f$ has \defn{defect} $d$ if every $f$-monochromatic component has maximum degree at most $d$ (that is, each vertex has monochromatic degree at most $d$). There have been many recent papers on clustered and defective colouring \citep{NSSW19,vdHW18,KO19,CE19,EJ14,EO16,DN17,LO18,HW19,MRW17,LW1,LW2,LW3,LW4,LW5,NSW22,DS20}; see \citep{WoodSurvey} for a survey. 
	
	The general goal of this paper is to study defective and clustered chromatic number of graph products,  with the focus on minimising the number of colours with bounded defect or bounded clustering a secondary goal. 
	
	The \defn{clustered chromatic number} of a graph class $\GG$, denoted by $\cchi(\GG)$, is the minimum integer $k$ for which there exists an integer $c$ such that every graph in $\GG$ has a $k$-colouring with clustering $c$. If there is no such integer $k$, then  $\GG$ has \defn{unbounded} clustered chromatic number. The \defn{defective chromatic number} of a graph class $\GG$, denoted by $\dchi(\GG)$, is the minimum integer $k$ for which there exists $c\in\NN$ such that every graph in $\GG$ has a $k$-colouring with defect $c$. If there is no such integer $k$, then  $\GG$ has \defn{unbounded} defective chromatic number.  Every colouring of a graph with clustering $c$ has defect $c-1$. Thus $\dchi(\GG)\leq \cchi(\GG) \leq\bigchi(\GG)$ for every class $\GG$. 
	
	Obviously, for all graphs $G$ and $H$,
	$$\max\{ \chi(G), \chi(H) \} \leq \chi( G \boxtimes H ) \leq \chi( G) \, \chi( H ) .$$
	The upper bound is tight if $G$ and $H$ are complete graphs, for example. 
	The lower bound is tight if $G$ or $H$ has no edges. 
	However, \citet{Vesztergombi78} proved that $\chi(G\boxtimes K_2)\geq \chi(G)+2$, implying that if $E(H)\neq\emptyset$ then 
	$$ \chi( G \boxtimes H ) \geq \chi(G) + 2 .$$
	More generally, \citet{KM94} proved  that 
	$$\chi( G \boxtimes H ) \geq \chi(G) + 2\omega(H) -2 .$$
	\citet{Zerovnik06} studied the chromatic numbers of the strong product
	of odd cycles.
	
	A classical result of \citet{Sabidussi57} states that for any graphs
	$G$ and $H$, $\chi(G\Box H)=\max\{\chi(G),\chi(H)\}$, while
	a famous conjecture of Hedetniemi stated that
	$\chi(G\times H)=\min\{\chi(G),\chi(H)\}$. This conjecture was
	recently disproved by \citet{Shitov19}. The remainder of the
	paper focuses on the strong product $\boxtimes$ of graphs rather
	than $\Box$ and $\times$.
	
	\subsection{Subgraphs of Strong Products}
	\label{SubgraphsStrongProducts}
		
	The study of colourings of strong products is partially motivated from
	the following results that show that natural classes of graphs are
	subgraphs of certain strong products. Thus, colouring results for the
	product imply an analogous result for the original class. Later in the paper we use \cref{DegreeTreewidthStructure},
	while \cref{KL,PlanarPartition} are not used. Nevertheless, these results provide further motivation for studying colouring of strong graph products, since they
	show that several classes with a complicated structure can be expressed
	as subgraphs of the strong product of significantly simpler graph classes.
	
	For a graph $G$ and an integer $d\ge 1$, let $\boxtimes_dG$ denote the $d$-fold strong
	product $G\boxtimes \dots \boxtimes G$.
	
	\begin{thm}[\citep{KL07}] 
		\label{KL}
		For every $c\in \NN$ there exists $d\in O(c\log c)$, such that if $G$ is a graph with $|\{w\in V(G):\dist(v,w)\leq r\}| \leq r^c$ for every vertex $v\in V(G)$ and integer  $r\geq 2$, then $G \subseteq \boxtimes_d P$. 
	\end{thm}
	
	\begin{thm}[\citep{DO95,Wood09,DW22}] 
		\label{DegreeTreewidthStructure}
		Every graph with maximum degree $\Delta\in\mathbb{N}^+$ and treewidth less than $k\in\mathbb{N}^+$ is a subgraph of $T\boxtimes K_{20k\Delta}$ for some tree $T$ with maximum degree at most $20k\Delta^2$. 
	\end{thm}
	
	\begin{thm}[\citep{DJMMUW20,UWY22}]
		\label{PlanarPartition}
		Every planar graph is a subgraph of:
		\begin{compactenum}[(a)]
			\item $H\boxtimes P$ for some planar graph $H$ of treewidth at most $6$ and for some path $P$;
			\item $H\boxtimes P\boxtimes K_3$ for some planar graph $H$ of treewidth at most $3$ and for some path $P$. 
		\end{compactenum}
	\end{thm}
	
	The interested reader is referred to \citep{DJMMUW20,DMW,DHHW22,HW21b,ISW,UTW} for extensions of this result to graphs of bounded genus, and other
	natural generalisations of planar graphs.

	\subsection{Hex Lemma}\label{sec:hex}
	
	The famous Hex Lemma says that the game of Hex cannot end in a draw; see \citep{HT19} for an account of the rich history of this game. 
	As illustrated in \cref{Hex}, the Hex Lemma is equivalent to saying that in every 2-colouring of the vertices of the $n\times n$ triangulated grid, there is a monochromatic path from one side to the opposite side.
	
	\begin{figure}[!h]
		\centering\includegraphics{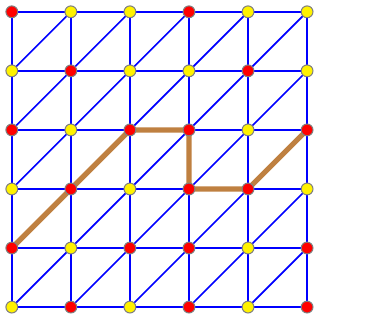}
		\caption{A Hex game.}
		\label{Hex}
	\end{figure}
	
	This result generalises to higher dimensions as follows. Let $G_n^d$ be the graph with vertex-set $\{1,\dots,n\}^d$, where distinct vertices $(v_1,\dots,v_d)$ and $(w_1,\dots,w_d)$ are adjacent in $G^d_n$ whenever $w_i\in\{v_i,v_i+1\}$ for each $i\in\{1,\dots,d\}$, or  $v_i\in\{w_i,w_i+1\}$ for each $i\in\{1,\dots,d\}$.
	Note that if each vertex $(v_1,\dots,v_d)$ is coloured $(\sum_i v_i)\bmod{(d+1)}$, then adjacent vertices $(v_1,\dots,v_d)$ and $(w_1,\dots,w_d)$ are assigned distinct colours, since $|(\sum_i v_i)-(\sum_i w_i)| \leq d$. Thus $\bigchi(G_n^d)\leq d+1$. In fact, $\bigchi(G^d_n)=d+1$ since $\{(v_1,\dots,v_d),(v_1+1,v_2,\dots,v_d),(v_1+1,v_2+1,v_3,\dots,v_d),\dots,(v_1+1,v_2+1,\dots,v_d+1)\}$ is a $(d+1)$-clique. The $d$-dimensional Hex Lemma provides a stronger lower bound: in every $d$-colouring of $G_n^d$ there is a monochromatic path from one `side' of $G^d_n$ to the opposite side \citep{Gale79}. Thus $$\cchi(\{ G^d_n: n \in \NN\})=d+1.$$ 
	See \citep{Gale79,LMST08,MP08,BDN17,MP08,Matdinov13,Karasev13} for related results. For example, \citet{Gale79} showed that this theorem is equivalent to the Brouwer Fixed Point Theorem. 
	
	These results are related to clustered colourings of strong products,
	as we now explain. Let $P_n$ denote the $n$-vertex path and  $\boxtimes_d P_{n}$ be the $d$-dimensional grid $P_n\boxtimes \cdots \boxtimes P_n$. Then $G_n^d$ is a subgraph of $\boxtimes_d P_{n}$. So 
	$$\cchi(\{ \boxtimes_d P_{n}  : n \in\NN\})\geq 
	\cchi(\{ G^d_n : n \in\NN\}) = d+1.$$ 
	A corollary of our main result is that equality holds here. In particular, \cref{HexGrid} shows that there is a $(d+1)$-colouring of $P^{\boxtimes d}_{n}$ with clustering $d!$. Thus
	$$\cchi(\{ \boxtimes_d P_{n} : n \in\NN\}) = d+1.$$
	
	\subsection{Asymptotic Dimension}\label{sec:asdim}
	
	Given a graph $G$ and an integer $\ell\ge 1$, $G^\ell$ is the graph
	obtained from $G$ by adding an edge between each pair of distinct vertices $u,v$
	at distance at most $\ell$ in $G$. Note that $G^1=G$. We say that a
	subset $S$ of vertices of a graph $G$ has \defn{weak diameter at most
		$d$} in $G$ if any two vertices of $S$ are at distance at most $d$
	in $G$.
	
	The asymptotic dimension of a metric space was introduced by
	\citet{Gro93} in the context of geometric group theory. For graph
	classes (and their shortest paths metric), it can be defined as
	follows~\citep{BBEGLPS}: the \defn{asymptotic dimension } of a graph
	class $\mathcal{F}$ is the minimum  $m\in\NN_0$ for which there is a function $f: \NN
	\rightarrow \NN$ such that for every $G \in \mathcal{F}$ and $\ell \in
	{\mathbb N}$, $G^\ell$ has an $(m+1)$-colouring in which each
	monochromatic component has  weak diameter at most $f(\ell)$ in
	$G^\ell$. (If no such integer $m$ exists, the asymptotic dimension of
	$\mathcal{F}$ is said to be $\infty$).
	
	Taking $\ell= 1$, we see that graphs from any graph class $\mathcal{F}$ of asymptotic dimension at most $m$ have an $(m+1)$-colouring in which each monochromatic component has bounded
	weak diameter. If, in addition, the graphs in $\mathcal{F}$ have bounded maximum degree, then all graphs in $\mathcal{F}$ have $(m+1)$-colourings with bounded clustering \citep{BBEGLPS}, implying $\cchi(\mathcal{F})\leq m+1$. 
	
	It is well-known  that the class of $d$-dimensional grids (with or
	without diagonals) has asymptotic dimension $d$ (see~\citep{Gro93}), and since they also
	have bounded degree,  it directly follows from the remarks above that $d$-dimensional grids have $(d+1)$-colourings with bounded clustering.
	
	An important problem is to bound the dimension of the product of topological or metric spaces as a function of their dimensions. It follows from the work of \citet{BD06} and \citet{BDLM08} that if
	$\mathcal{F}_1$ and $\mathcal{F}_2$ are classes of asymptotic
	dimension $m_1$ and $m_2$, respectively, then the class $\mathcal{F}_1 \boxtimes \mathcal{F}_2 := \{G_1\boxtimes G_2\,:\, G_1\in\mathcal{F}_1, G_2\in \mathcal{F}_2\}$ has asymptotic dimension at most $m_1+m_2$. For example, that $d$-dimensional grids have asymptotic dimension at most $d$ can be deduced from this product theorem by induction, using the fact that the family of paths has asymptotic dimension 1.
	In particular, if two classes $\mathcal{F}_1$ and
	$\mathcal{F}_2$ have asymptotic dimension $m_1$ and $m_2$,
	respectively, and have uniformly bounded maximum degree, then the 
	graphs in $ \mathcal{F}_1 \boxtimes \mathcal{F}_2$ have $(m_1+m_2+1)$-colourings
	with bounded clustering. Since  graphs of bounded treewidth have asymptotic dimension at most 1 \citep{BBEGLPS}, this implies the following.
	
	\begin{restatable}{thm}{DegreeTreewidthClustered}
		\label{DegreeTreewidthClustered}
		If $G_1,\dots,G_d$ are graphs with treewidth at most $k\in\NN$ and maximum degree at most $\Delta\in\NN$, then $G_1\boxtimes \dots\boxtimes G_d$ is $(d+1)$-colourable with clustering at most some function $c(d,\Delta,k)$.
	\end{restatable}
	
	Similarly, using the fact that graphs excluding some fixed minor have asymptotic dimension at most 2
	\citep{BBEGLPS}, we have the following.
	
	\begin{thm}
		\label{DegreeMinorClustered}
		Let $H$ be a graph. If $G_1,\dots,G_d$ are $H$-minor free graphs with maximum degree at most $\Delta\in\NN$, then $G_1\boxtimes \dots\boxtimes G_d$ is $(2d+1)$-colourable with clustering at most some function $c(d,\Delta,H)$.
	\end{thm}

	The conditions that $\mathcal{F}_1$ and $\mathcal{F}_2$ have bounded asymptotic dimension
	and degree are quite strong, and instead we would like to obtain
	conditions only based on the fact that $\mathcal{F}_1$ and $\mathcal{F}_2$
	are themselves colourable with bounded clustering with few colours, and if possible, without the maximum degree assumption.

	\subsection{Fractional Colouring}\label{sec:frac}
	
	Let $G$ be a graph.  For $p,q\in\NN$ with $p\geq q$, a \defn{$(p\!:\!q)$-colouring} of $G$ is a function $f:V(G)\to \binom{C}{q}$ for some set $C$ with $|C|=p$. That is, each vertex is assigned a set of $q$ colours out of a palette of $p$ colours. 
	For $t\in\RR$, a \defn{fractional $t$-colouring} is a $(p\!:\!q)$-colouring for some $p,q\in\NN$ with $\frac{p}{q}\leq t$. 
	A \defn{$(p\!:\!q)$-colouring} $f$ of $G$ is \defn{proper} if $f(v)\cap f(w)=\emptyset$ for each edge $vw\in E(G)$. 
	
	The \defn{fractional chromatic number} of $G$ is 
	$$\fchi(G) := \inf\left\{ t \in\RR \,: \, \text{$G$ has a proper fractional $t$-colouring} \right\}.$$
	The fractional chromatic number is widely studied; see the textbook \citep{SU97}, which includes a proof of the fundamental property that $\fchi(G)\in\QQ$. 
	
	The next result relates $\fchi(G)$ and $\alpha(G)$, the size of the largest independent	set in $G$. 
	
	\begin{lem}[\citep{SU97}] 
		\label{fchiAlphaVertexTransitive}
		For every graph $G$,
		$$\fchi(G) \, \alpha(G) \geq |V(G)|,$$
		with equality if $G$ is vertex-transitive.
	\end{lem}
	
	
	%
	%
	%
	%
	
	The following well-known observation shows an immediate connection between fractional colouring and strong products. 
	
	\begin{obs}
		\label{pqProduct}
		A graph $G$ is properly $(p\!:\!q)$-colourable if and only if $G\boxtimes K_q$ is properly $p$-colourable.
	\end{obs}
	
	\cref{pqProduct} is normally stated in terms of the lexicographic product $G[K_q]$, which equals $G\boxtimes K_q$ (although $G[H]\neq G\boxtimes H$ for other graphs $H$). 
	See \citep{KY02,Klavzar98} for results on the fractional chromatic number and the lexicographic product.

	%
	%
	%
	%
	%

	Fractional 1-defective colourings were first studied by \citet{FS15}; see \citep{GX16,MOS11,Klostermeyer02} for related results. Fractional clustered colourings were introduced by \citet{DS20} and subsequently studied by \citet{NSW22} and \citet{LW5}. 

The notions of clustered and defective colourings introduced
in \cref{sec:def} naturally extend to fractional colouring as follows. For a $(p\!:\!q)$-colouring $f:V(G)\to \binom{C}{q}$ of $G$ and for each colour $\alpha\in C$, the subgraph $G[ \{ v \in V(G): \alpha \in f(v) \} ]$ is called an \defn{$f$-monochromatic subgraph} or \defn{monochromatic subgraph} when $f$ is clear from the context. A connected component of an $f$-monochromatic subgraph is called an \defn{$f$-monochromatic component} or \defn{monochromatic component}. Note that $f$ is proper if and only if each $f$-monochromatic component has exactly one vertex. 
	
	A \defn{$(p\!:\!q)$-colouring} has \defn{defect} $c$ if every monochromatic subgraph has maximum degree at most $c$. A $(p\!:\!q)$-colouring has \defn{clustering} $c$ if every monochromatic component has at most $c$ vertices. 
	
	The \defn{fractional clustered chromatic number $\cfchi(\GG)$} of a graph class $\GG$ is the infimum of all $t\in\RR$ such that, for some $c\in\NN$, every graph in $\GG$ is fractionally $t$-colourable with clustering $c$. The \defn{fractional defective chromatic number $\dfchi(\GG)$} of a graph class $\GG$ is the infimum of $t>0$ such that, for some $c\in\NN$, every graph in $\GG$ is fractionally $t$-colourable with defect $c$.	
	
	\citet{Dvorak16} proved that every hereditary class admitting
	strongly sublinear separators and bounded maximum degree has
	fractional clustered chromatic number 1 (see also \cite{DS20}). Using, this result, \citet{LW5} proved that for every hereditary graph class $\GG$ admitting strongly sublinear separators, $$\cfchi(\GG)=\dfchi(\GG).$$ 
	\citet{LW5} also proved that for every monotone graph class $\GG$  admitting strongly sublinear separators and with $K_{s,t}\not\in\GG$, 
	$$\cfchi(\GG) = \dfchi(\GG) \leq \dchi(\GG) \leq s.$$
	
	\citet{NSW22} determined $\dfchi$ and $\cfchi$ for every minor-closed class. In particular, for every proper minor-closed class $\GG$, 
	$$\dfchi(\GG)=\cfchi(\GG)=\min\{k\in\NN:\exists
	n\,C_{k,n}\not\in\GG\},$$ where $C_{n,k}$ is a specific graph (see
	\cref{sec:ptc} for the definition of $C_{n,k}$ and for more details about this result).
	As an example, say $\GG_t$ is the class of $K_t$-minor-free graphs. \citet{Hadwiger43} famously conjectured that $\bigchi(\GG_t)=t-1$. It is even open whether $\fchi(\GG_t)=t-1$. 
	The best known upper bound is $\fchi(\GG_t) \leq 2t-2$ due to \citet{RS98}.
	\citet{EKKOS15} proved that $$\dchi(\GG_t)=t-1.$$
	It is open whether $\cchi(\GG_t)=t-1$. 
	The best known upper bound is $\cchi(\GG_t)\leq t+1$ due to \citet{LW2}. \citet{DN17} have announced that a forthcoming paper will prove that $\cchi(\GG_t)=t-1$. 
	The above result of \citet{NSW22} implies that 
	$$\dfchi(\GG_t)=\cfchi(\GG_t)=t-1.$$
	As another example, the result of \citet{NSW22} implies that the class of graphs embeddable in any fixed surface has fractional clustered chromatic number and fractional defective chromatic number 3.
	
	\subsection{Site percolation}\label{sec:perco}
	
	There is an interesting connection between percolation
	and fractional clustered colouring. Consider a graph $G$ and a real number $x>0$, and let $S$ be a
	random subset of vertices of $G$, such that $\Pr(v\in S)\ge x$ for every $v\in V(G)$. Then $S$ is a \defn{site percolation} of
	density at least $x$. If the events $(v\in S)_{v\in V(G)}$ are
	independent and $\Pr(v\in S)= x$ for every $v\in V(G)$, then $S$ is called
	a \defn{Bernouilli site percolation}, but in
	general the events $(v\in S)_{v\in V(G)}$  can be dependent. Each connected component of
	$G[S]$ is called a \defn{cluster} of $S$, and $S$ has
	\defn{bounded clustering} (for a family of graphs $G$) if all clusters have bounded size.
	An important problem in percolation theory is to understand when
	$S$ has finite clusters (when $G$ is infinite), or when $S$ has
	bounded clustering, independent of the size of $G$ (when $G$ is
	finite). Assume that all clusters of $S$ have bounded
	size almost surely. Then, by discarding the vanishing proportion of sets of
	unbounded clustering in the support of $S$, we obtain a probability distribution
	over the subsets of vertices of bounded clustering in $G$, such that
	each $v\in V(G)$ is in a random subset (according to the distribution)
	with probability at least $x-\epsilon$, for any $\epsilon>0$. If
	all graphs $G$ in some class $\GG$ satisfy this property (with
	uniformly bounded clustering), this implies that $\cfchi(\GG)\le
	\tfrac1{x-\epsilon}$, for any $\epsilon>0$, and thus  $\cfchi(\GG)\le
	\tfrac1{x}$. Conversely, if a class $\GG$ satisfies $\cfchi(\GG)\le
	\tfrac1{x}$, then this clearly gives a site percolation of bounded
	clustering for $\GG$, with density at least $x$.
	
	As an example, \citet{CGHV15} recently proved that in any cubic
	graph of sufficiently large (but constant) girth, there is a percolation of density at least
	$0.534$ in which all clusters are bounded almost surely. It follows
	that for this class of graphs, $\cfchi(\GG)\le \tfrac1{0.534}\le
	1.873$.
	
	Note that percolation in finite dimensional lattices (and in
	particular the critical probability at which an infinite cluster
	appears) is a well-studied topic in probability theory. Finite
	dimensional lattices themselves are easily expressed as a strong
	product of paths.


	\subsection{Shannon Capacity}\label{sec:shannon}
	
	Motivated by connections to communications theory, \citet{Shannon56}
	defined what is now called the \defn{Shannon capacity} of  a graph $G$
	to be  $$\Theta(G)=\sup_{d}(\alpha (\boxtimes_d G))^{1/d}=\lim_{d\rightarrow \infty } (\alpha (\boxtimes_d G))^{1/d}.$$
	 \citet{Lovasz79} famously proved that $\Theta(C_5)=\sqrt{5}$. See
	\citep{AL06,Alon02,Alon98,Vesztergombi78,KM94,Vesztergombi81,Farber86,HR82,KM94,Klavzar93}
	for more results.
	
	By \cref{fchiAlphaVertexTransitive}, for any graph $G$ we have
	$\alpha(G) \ge |V(G)|/\fchi(G)$, with equality if $G$ is
	vertex-transitive. It follows that  for any integer $d\ge 1$, $$\alpha(\boxtimes_d G) \ge |V(\boxtimes_d G)|/\fchi(\boxtimes_d
	G)=|V(G)|^d/\fchi(\boxtimes_d G),$$ 
	with equality if $\boxtimes_d G$ is
	vertex-transitive, and in particular if $G$ itself is vertex-transitive.
	As a consequence, we have the
	following alternative definition of the Shannon capacity of
	vertex-transitive graphs in terms of the fractional chromatic number
	of their strong products.
	
	\begin{obs}
		\label{ShannonReformulation}
		For any graph $G$,  $$\Theta(G)\ge \frac{|V(G)|}{\inf_{d } (\fchi (\boxtimes_d G))^{1/d}}=\frac{|V(G)|}{\lim_{d\rightarrow
				\infty } (\fchi (\boxtimes_d G))^{1/d}},$$ with equality if $G$ is vertex-transitive.
	\end{obs}
	
	As a consequence, results on the Shannon capacity of graphs imply
	lower bounds (or exact bounds) on the fractional chromatic number of
	strong products of graphs.
	
	\subsection{Our results}\label{sec:results}
	
	We start by recalling basic results on the chromatic number of
	the product $G_1\boxtimes \cdots\boxtimes G_d$ in
	\cref{ProductColourings}: the chromatic number of the product is at
	most the product of the chromatic numbers. We show that the same holds
	for the fractional version and for the clustered version (for graph classes). While complete graphs show
	that the result on proper colouring is tight in general, other
	constructions are needed for (fractional) clustered chromatic
	number. In \cref{sec:ptc}, we show that for products of tree-closures, the
	(fractional) defective and clustered chromatic number is equal to the
	product of the (fractional) defective and clustered chromatic numbers.
	
	In \cref{sec:consistent}, we introduce consistent
	(fractional) colouring and use it to combine any proper
	$(p\!:\!q)$-colouring of a graph $G$ with a $(q\!:\!r)$-colouring of a
	graph $H$ with
	bounded clustering into a $(p\!:\!r)$-colouring of $G\boxtimes H$ of
	bounded clustering. Using consistent $(k+1\!:\!k)$-colourings of paths
	(and more generally bounded degree trees) with bounded clustering, we prove general results on the clustered chromatic number of the product of a graph
	with a path (and more generally a bounded degree tree, or a graph of
	bounded treewidth and maximum degree). We also study the fractional
	clustered chromatic number of graphs of bounded degree, showing that
	the best known lower bound for the non-fractional case also holds for
	the fractional relaxation.
	
	In \cref{sec:param}, we prove that many of our results on the clustered
	colouring of graph products can be extended to the broader setting of
	general graph parameters.
	
	\section{Basics}
	
	\subsection{Product Colourings}
	\label{ProductColourings}

	We start with the following folklore result about proper colourings of strong products; see the informative survey about proper colouring of graph products by \citet{Klavzar96}. 
	
	\begin{lem}
		\label{ChromaticNumber}
		For all graphs $G_1,\dots,G_d$, 
		$$\chi( G_1\boxtimes \cdots\boxtimes G_d) \leq \prod_{i=1}^d \chi(G_i).$$
	\end{lem}
	
	\begin{proof}
		Let $\phi_i$ be a $\chi(G_i)$-colouring of $G_i$. Assign each vertex $(v_1,\dots,v_d)$ of $G_1\boxtimes \cdots\boxtimes G_d$ the colour $(\phi_1(v_1),\dots,\phi_d(v_d))$. If $(v_1,\dots,v_d)(w_1,\dots,w_d)$ is an edge of $G_1\boxtimes \cdots\boxtimes G_d$, then $v_iw_i$ is an edge of $G_i$ for some $i$, implying $\phi_i(v_i)\neq \phi_i(w_i)$ and 
		$G_1\boxtimes \cdots\boxtimes G_d$ is properly coloured with $\prod_{i=1}^d \chi(G_i)$ colours. 
	\end{proof}
	
	We have the following similar result for fractional colouring. 
	
	\begin{lem}
		\label{FractionalProduct}
		For all graphs $G_1,\dots,G_d$, 
		$$\fchi( G_1\boxtimes \cdots\boxtimes G_d) \leq \prod_{i=1}^d \fchi(G_i).$$
	\end{lem}
	
	\begin{proof}
		$\fchi(G_i)=\frac{p_i}{q_i}$ for some $p_i,q_i\in\NN$. By \cref{pqProduct}, $\chi(G_i\boxtimes K_{q_i} )\leq p_i$. 
		Let $P:=\prod_i p_i$ and $Q:=\prod_i q_i$. 
		By \cref{ChromaticNumber}, 
		$\chi( ( G_1\boxtimes K_{q_1} ) \cdots ( G_d \boxtimes K_{q_d} ) ) \leq P.$
		Since $( G_1\boxtimes K_{q_1} ) \cdots ( G_d \boxtimes K_{q_d} ) \cong ( G_1\boxtimes \cdots \boxtimes G_d) \boxtimes K_Q$, we have
		$\chi( ( G_1\boxtimes \cdots \boxtimes G_d ) \boxtimes K_Q ) \leq P$. 
		By \cref{pqProduct} again, 
		$G_1\boxtimes \cdots \boxtimes G_d $ is $(P\!:\!Q)$-colourable, and 
		$\fchi( G_1\boxtimes \cdots \boxtimes G_d ) \leq P/Q = \prod_{i=1}^d \fchi(G_i)$.
	\end{proof}
	
	Equality holds in \cref{FractionalProduct} when $G_1,\dots,G_d$ are
	complete graphs, for example. However, equality does not always hold
	in \cref{FractionalProduct}. For example, if $G=H=C_5$ then
	$\fchi(C_5)=\frac{5}{2}$ but $\fchi(C_5\boxtimes C_5)\leq \chi(C_5
	\boxtimes C_5) \leq 5$, where a proper 5-colouring of $C_5\boxtimes
	C_5$ is shown in \cref{C5C5}. In fact, a simple case-analysis shows
	that $\alpha(C_5\boxtimes C_5)=5$, implying that $\fchi(C_5\boxtimes
	C_5)=\bigchi(C_5\boxtimes C_5)=5$ (since $\alpha(G)\, \fchi(G) \geq
	|V(G)|$ for every graph $G$). Note that using \cref{ShannonReformulation}, the
	classical result of \citet{Lovasz79} stating that
	$\Theta(C_5)=\sqrt{5}$ can be rephrased as $\fchi(\boxtimes_d
	C_5)=5^{d/2}$ for any $d\ge 2$.
	
	\begin{figure}[!h]
		\centering 
		\includegraphics{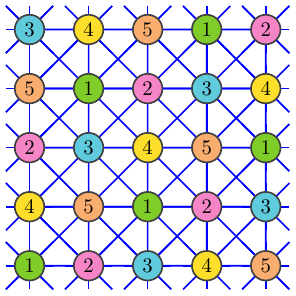}
		\caption{A proper 5-colouring of $C_5\boxtimes C_5$}
		\label{C5C5}
	\end{figure}

	
	By \cref{FractionalProduct},  for all graphs $G_1$ and $G_2$, 
	$$\fchi(G_1 \boxtimes G_2) \le \fchi(G_1) \fchi(G_2),$$ with equality when $G_1$ or $G_2$ is a complete graph~\citep{KY02}\footnote{\citet{KY02} proved that $\fchi(G_1 \circ G_2) = \fchi(G_1) \fchi(G_2)$ for all graphs $G_1$ and $G_2$, where $\circ$ denotes the lexicographic product. Since $G\boxtimes K_t = G \circ K_t$, we have $\fchi(G \boxtimes K_t) = t\, \fchi(G)$ for every graph $G$.}. It is tempting therefore to hope for an analogous lower bound on $\fchi(G_1 \boxtimes G_2)$ in terms of $\fchi(G_1)$ and $\fchi(G_2)$ for all graphs $G_1$ and $G_2$. The following lemma dashes that hope. 
	
	\begin{lem}
		For infinitely many $n\in\NN$ there is an $n$-vertex graph $G$ such that 
		$$\fchi(G \boxtimes G) \le \frac{ (256+o(1))\log^4n}{n} \, \fchi(G)^2.$$
	\end{lem}
	
	\begin{proof}
		\citet[Theorems~5~and~6]{AO95} proved that for infinitely many $n\in\NN$, there is a (Cayley) graph $G$ on $n$ vertices with $\chi(G \boxtimes G)\le n$ and $\alpha(G)\leq (16+o(1))\log^2 n$. By \cref{fchiAlphaVertexTransitive}, 
		\begin{align*}
			\fchi(G \boxtimes G) 
			\leq \chi(G\boxtimes G)
			\leq n
			\leq\;  & \frac{ (256+o(1))\log^4n}{n} \left(\frac{n}{\alpha(G)}\right)^2 \\
			\le\; &  
			\frac{ (256+o(1))\log^4n}{n}\,\fchi(G)^2.\qedhere
		\end{align*}
	\end{proof}

	The next lemma generalises \cref{ChromaticNumber,FractionalProduct}. 
	
	\begin{lem}
		\label{ClusteredProductColouring}
		Let $G_1,\dots,G_d$ be graphs, such that $G_i$ is
		$(p_i\!:\!q_i)$-colourable with clustering $c_i$, for each
		$i\in[1,d]$. Then $G:= G_1\boxtimes \cdots\boxtimes G_d$ is $(\prod_i
		p_i:\prod_i q_i)$-colourable with clustering $\prod_ic_i$.
	\end{lem}
	
	\begin{proof}
		For $i\in[1,d]$, let $\phi_i$ be a $(p_i\!:\!q_i)$-colouring of $G_i$ with clustering $c_i$.
		Let $\phi$ be the colouring of $G$, where each vertex
		$v=(v_1,\dots,v_d)$ of $G$ is coloured $\phi(v) := \{ (a_1,\dots,a_d):
		a_i \in \phi_i(v_i), i \in [1,d] \}$. So each vertex of $G$ is
		assigned a set of $\prod_i q_i$ colours, and there are $\prod_i p_i$
		colours in total. Let $X:= X_1\boxtimes\dots\boxtimes X_d$, where each
		$X_i$ is a monochromatic component of $G_i$ using colour $a_i$. Then
		$X$ is a monochromatic connected induced subgraph of $G$ using colour
		$(a_1,\dots,a_d)$. Consider any edge $(v_1,\dots,v_d)(w_1,\dots,w_d)$
		of $G$ with $(v_1,\dots,v_d)\in V(X)$ and $(w_1,\dots,w_d)\not\in
		V(X)$. Thus $v_iw_i\in E(G_i)$ and  $w_i\not\in V(X_i)$ for some
		$i\in\{1,\dots,d\}$. Hence $a_i\not\in \phi_i(w_i)$, implying
		$(a_1,\dots,a_d)\not\in \phi(w)$. Hence $X$ is a monochromatic
		component of $G$ using colour $(a_1,\dots,a_d)$. As $X$ contains at
		most $\prod_ic_i$ vertices, it follows that $\phi$ has clustering at
		most $\prod_ic_i$.
	\end{proof}

	\cref{ClusteredProductColouring} implies the following analogues of \cref{ChromaticNumber} for clustered colouring.

	\begin{thm}
		\label{ClusteredProductColouringClasses}
		For all graph classes $\GG_1,\dots,\GG_d$ 
		$$\cchi( \GG_1 \boxtimes \dots\boxtimes \GG_d) \leq \prod_{i=1}^d \cchi(\GG_i).$$
	\end{thm}

	It is interesting to consider when the naive upper bound in
	\cref{ClusteredProductColouringClasses} is tight.
	\cref{FracClosureProduct} below shows that this result is tight
	for products of closures of high-degree trees.

	Finally, note that \cref{ClusteredProductColouring} implies the following basic observation.
	
	\begin{lem}
		\label{BlowUp}
		If a graph $G$ is $k$-colourable with clustering $c$, then $G\boxtimes K_t$ is $k$-colourable with clustering $ct$.
	\end{lem}
	
	More generally, \cref{ClusteredProductColouringClasses} implies:
	
	\begin{lem}
		\label{FractionalBlowUp}
		If a graph $G$ is $(p\!:\!q)$-colourable with clustering $c$, then $G\boxtimes K_t$ is $(p\!:\!q)$-colourable with clustering $ct$.
	\end{lem}
	
	\section{Products of Tree-Closures}\label{sec:ptc}
	
	This section presents examples of graphs that show that the naive upper bound in  \cref{ClusteredProductColouringClasses} is tight. 

        The \emph{depth} of a node in a rooted tree is its distance to
        the root plus one.
	For $k,n\in\NN$, let $T_{k,n}$ be the rooted tree in which
        every leaf is at depth $k$, and every non-leaf has $n$ children. Let $C_{k,n}$ be the graph obtained from $T_{k,n}$ by adding an edge between every ancestor and descendant (called the \defn{closure} of $T_{k,n}$). Colouring each vertex by its distance from the root gives a $k$-colouring of $C_{k,n}$, and any root-leaf path in $C_{k,n}$ induces a $k$-clique. So $\chi(C_{k,n})=k$.
	
	The class $\CC_k:= \{ C_{k,n} : n \in \NN\}$ is important for
        defective and clustered colouring, and is often called the
        `standard' example. It is well-known and easily proved (see
        \citep{WoodSurvey})
        that $$\dchi(\CC_k)=\cchi(\CC_k)=\bigchi(\CC_k)=k.$$

	\citet{NSW22} extended this result (using a result of \citet{DS20}) to the setting of defective and clustered fractional chromatic number by showing that
	$$\dfchi(\CC_k)= \cfchi(\CC_k) = \bigchi^f(\CC_k) = \dchi(\CC_k)=\cchi(\CC_k)=\bigchi(\CC_k)=k.$$
	Here we give an elementary and self-contained proof of this
        result. In fact, we prove the following generalisation in
        terms of strong products.  This shows that
        \cref{ClusteredProductColouringClasses} is tight,
        even for fractional colourings. 

	\begin{thm}
		\label{FracClosureProduct}
		For all $k,d\in\NN$, if $\GG := \boxtimes_d \CC_{k}$ then 
		$$ \dfchi(\GG) = \cfchi(\GG) = \fchi(\GG) =\dchi( \GG ) = \cchi( \GG ) = \bigchi( \GG ) =k^d.$$
    \end{thm}

\begin{proof}
Let $K:=k^d$. It follows from the definitions that 
		$\dfchi(\GG)\leq \dchi(\GG) \leq \chi(\GG)=K$ and 
		$\dfchi(\GG)\leq \cfchi(\GG) \leq \cchi(\GG) \leq \chi(\GG)=K$ and
		$\dfchi(\GG)\leq \bigchi^f(\GG) \leq \chi(\GG)=K$. 
		Thus it suffices to prove that $\dfchi(\GG)\geq K$. Recall that $\dfchi(\GG)$ is the infimum of all $t\in\RR$ such that, for some $c\in\NN$, for every $G\in\GG$ there exists $p,q\in\NN$ such that $p\leq tq$ and $G$ is $(p\!:\!q)$-colourable with defect $c$. 
		
		Suppose for the sake of contradiction that $\dfchi(\GG)\le K-\epsilon$, for some $\epsilon>0$. It follows that there
                exists $c\in\NN$ such that for every integer
                $n$ there exist $p,q\in\NN$ such that $p\leq
                (K-\epsilon)q$ and $\boxtimes_d C_{k,n}$ is $(p\!:\!q)$-colourable with defect $c$.

For $n\in \NN$, consider the graph $G=\boxtimes_d C_{k,n}$ and a
$(p\!:\!q)$-colouring of $G$ with defect $c$ such that $p\leq
(K-\epsilon)q$. We will prove that for fixed $k$, $d$, and $\epsilon$, the value of $c$ must be at least linear in $n$. Since $n$ can be chosen to be arbitrary, this immediately yields the desired contradiction.

Each vertex $x$ of $G$ is a  $d$-tuple $(x_1,\ldots,x_d)$, where each
$x_i$ is a vertex of $C_{k,n}$. Whenever we mention ancestors,
descendants, leaves and the depth of vertices in $C_{k,n}$, these terms refer to the corresponding notions in the spanning subgraph $T_{k,n}$ of $C_{k,n}$. Note that distinct vertices
$x=(x_1,\ldots,x_d)$ and $y=(y_1,\ldots,y_d)$ are adjacent in $G$ if
and only if for each $i\in[d]$, $x_i$ is an ancestor of $y_i$ or
$y_i$ is an ancestor of $x_i$ in $C_{k,n}$  (where we adopt the convention
that every vertex is an ancestor and descendant of itself).

For each $k$-tuple of non-negative integers $s=(s_1,\ldots,s_k)$ such that $\sum_{i=1}^k s_i=d$, define $V_s$ to be the set of vertices $x=(x_1,\ldots,x_d)$ of $G$ such that for each
$i\in [k]$, there are precisely $s_i$ indices $j\in [d]$ such
that $x_j$ has depth $i$ in $C_{k,n}$. Since $C_{k,n}$ contains $n^{i-1}$
vertices at depth $i$ (for each $i\in [d]$) and since $\binom{a}{b}\leq 2^a$, 
$$|V_s|={d\choose s_1}{d-s_1\choose s_2}\cdots {d-s_1-\cdots -s_{k-1}\choose s_k}\cdot n^{s_2}n^{2s_3}\cdots n^{(k-1)s_k}\le 2^{dk}\cdot n^{\sum_{i=1}^k  (i-1)s_i}.$$

Let $V^*:= V_{(0,\ldots,0,d)}$; that is, $V^*$ is the
set of vertices $x=(x_1,\ldots,x_d)$ of $G$ such that $x_1,\ldots,x_d$ are leaves of $C_{k,n}$. Note that
$|V^*|=n^{d(k-1)}$.  

For each vertex $x=(x_1,\ldots,x_d)\in V^*$, let $Q(x)$ be the set of vertices $y=(y_1,\ldots,y_d)$ of  $G$
such that for each $i\in [d]$, $y_i$ is an ancestor of $x_i$ in
$C_{k,n}$. By the definition of $G$, $Q(x)$ is a clique of
size $k^d=K$ (including $x$). 
Let $S_1,\ldots,S_K$ be the sets of colours assigned to
the elements of $Q(x)$ (so each $S_i$ is a $q$-element subset of $[p]$, with $p\le (K-\epsilon)q$). We claim that there are indices $i<j$ such that $|S_i\cap S_j|\ge \tfrac{\epsilon}{K^2}\cdot q$. If not,  for each $i\in[K]$, $S_i$
has at least $q-(i-1)\tfrac{\epsilon}{K^2}\cdot q$ elements not in
$S_1,S_2,\ldots S_{i-1}$. Thus 
$$|\bigcup_{i=1}^K S_i| \geq 
\sum_{i=1}^K (q-(i-1)\tfrac{\epsilon}{K^2}\cdot q )
>Kq-\tfrac{K^2}2\cdot \tfrac{\epsilon}{K^2}\cdot
q
>(K-\epsilon)q\ge p,$$ 
which is a contradiction. Thus there exist distinct vertices $u,v\in Q(x)$ whose sets of colours  intersect in at least
$\tfrac{\epsilon}{K^2}\cdot q$ elements. Assume without loss of generality
that $u\in V_{s}$ and $v\in V_t$, and the sequence $s$ precedes $t$ in
lexicographic order. Orient the edge $uv$ from $u$ to $v$ and \emph{charge} $x$ to the arc $(u,v)$.

We now bound the number of vertices $x=(x_1,\ldots,x_d)\in V^*$ that
are charged to a given arc $(u,v)$, with $u=(u_1,\ldots,u_d)\in V_s$ and $v=(v_1\ldots,v_d)\in V_t$, where $s$ precedes $t=(t_1,\ldots,t_k)$ in lexicographic order. By definition, if $x$ is charged to $(u,v)$, each
$x_i$ is a descendant of $u_i$ and $v_i$ (and in particular $u_i$ and
$v_i$ are also in ancestor relationship). Each vertex at
depth $i$ in $C_{k,n}$ has precisely $n^{k-i}$ descendants that are
leaves in $C_{k,n}$. Thus (considering only $v_1,\dots,v_d$), 
at most 
$$ \prod_{i=1}^k n^{t_i(k-i)} =n^{\sum_{i=1}^kt_i(k-i)}$$ 
vertices of $V^*$ are charged to $(u,v)$.

We claim that there is an index $i\in [d]$, such that $u_i$ is a
strict descendant of $v_i$. Since $x_i$ is a descendant of $u_i$, it
follows that the bound above can be divided by a factor $n$, and thus at most $n^{-1+\sum_{i=1}^kt_i(k-i)}$ vertices of $V^*$ are charged to $(u,v)$. 
To prove the claim, consider first the $t_1$ indices $i\in [d]$ such that $v_i$ has depth 1. If $u_i$ has depth greater than 1 for one of these indices, then the desired property holds. So we may assume that all the corresponding $u_i$'s also have depth 1. Since $s$ precedes $t$ in lexicographic order, $s_1\le t_1$ and
thus $s_1=t_1$. It follows that for each $i\in [d]$, $u_i$ has depth 1
if and only if $v_i$ has depth 1. By considering the $t_2$ indices
$i$ such that $v_i$ has depth 2, the same reasoning shows that for
each $i\in [d]$, $v_i$ has depth 2 if and only if $u_i$ has depth
2. By iterating this argument, for each $i\in [d]$ and
$j\in [k]$, $v_i$ has depth $j$ if and only if $u_i$ has depth
$j$. Since $u_i$ and $v_i$ are ancestors for each $i\in [d]$, we have that
$u=v$, which is a contradiction. It follows that some $u_i$ is a strict ancestor of $v_i$, and thus (as argued above) at most $n^{-1+\sum_{i=1}^kt_i(k-i)}$ vertices of $V^*$ are charged to $(u,v)$.

Each vertex of $V^*$ is charged to some arc $(u,v)$, where $u$ and $v$
share at least $\tfrac{\epsilon}{K^2}\cdot q$ colours. We claim that
for each $v\in V^*$, there are at most $cK^2/\epsilon$ such arcs
$(u,v)$. If not, the $q$ colours of $v$ must appear (with repetition)
more than $\tfrac{\epsilon}{K^2}\cdot q \cdot cK^2/\epsilon=cq$ times
in the neighbourhood of $v$. By the pigeonhole principle some colour
of $v$ appears more than $c$ times in the neighbourhood of $v$, which
contradicts the assumption that the colouring has defect at most $c$.

For each vertex $v\in V_t$, where $t=(t_1,\ldots,t_k)$, and for each of the at most $cK^2/\epsilon$ arcs $(u,v)$ as above, we have proved that  at most  $n^{-1+\sum_{i=1}^kt_i(k-i)}       $ vertices of $V^*$ are
charged to $(u,v)$. Since $|V_t|\le 2^{dk}\cdot n^{\sum_{i=1}^k
  (i-1)t_i}$ and $\sum_{i=1}^kt_i=d$, at most $$c\tfrac{K^2}{\epsilon}\cdot2^{dk}\cdot n^{\sum_{i=1}^k
  (i-1)t_i}\cdot n^{-1+\sum_{i=1}^kt_i(k-i)} =c\tfrac{K^2}{\epsilon}\cdot 2^{dk}\cdot n^{-1+\sum_{i=1}^k(k-1)t_i}=c\tfrac{K^2}{\epsilon}\cdot2^{dk}\cdot n^{d(k-1)-1}$$ vertices of $V^*$ are charged to an arc $(u,v)$. 
Since there are at most $(d+1)^k$ possible $k$-tuples of integers
$t=(t_1,\ldots,t_k)$ with $\sum_{i=1}^kt_i=d$, it follows that
$$n^{d(k-1)}=|V^*|\le (d+1)^k c\cdot \tfrac{K^2}{\epsilon}\cdot2^{dk}\cdot n^{d(k-1)-1} ,$$ and
thus $c\ge n\cdot \tfrac{\epsilon}{k^{2d} (d+1)^k 2^{dk}}$, as desired.
\end{proof}

	Let $\STAR$ be the class of all star graphs; that is, $\STAR:=\{K_{1,n}:n\in\NN\}$. 
	Since $K_{1,n} \cong C_{2,n}$, \cref{FracClosureProduct} implies:
	
	\begin{cor}
		\label{StarProduct}
		For $d\in\mathbb{N}$, let $\STAR^d$ be the class of all $d$-dimensional strong products of star graphs. Then 
		$$\dfchi(\STAR^d) = \cfchi(\STAR^d) = \fchi(\STAR^d) = \dchi( \STAR^d) = \cchi( \STAR^d ) = \bigchi( \STAR^d ) = 2^d.$$
	\end{cor}

	\section{Consistent colourings}\label{sec:consistent}

	A $(p\!:\!q)$-colouring $\alpha$ of a graph $G$ is \defn{consistent} if for each vertex $x\in V(G)$, there is an ordering $\alpha_x^1,\dots,\alpha_x^q$ of $\alpha(x)$, such that $\alpha_x^{i} \neq \alpha_y^{j}$ for each edge $xy\in E(G)$ and for all distinct $i,j\in[1,q]$. For example, the following $(4\!:\!3)$-colouring of a path is consistent:
	$$\begin{array}{c} 0\\ 1\\ 2\end{array}, \begin{array}{c} 0\\ 1\\
		3\end{array}, \begin{array}{c}
		0\\ 2\\
		3\end{array}, \begin{array}{c}
		1\\
		2\\
		3\end{array}, \begin{array}{c}
		1\\
		2\\
		0\end{array}, \begin{array}{c}
		1\\
		3\\
		0\end{array}, \begin{array}{c}
		2\\
		3\\
		0\end{array}, \begin{array}{c} 2\\ 3\\ 1\end{array}, \begin{array}{c} 2\\ 0\\ 1\end{array}, \begin{array}{c} 3\\ 0\\ 1\end{array}, \begin{array}{c} 3\\ 0\\ 2\end{array}, \begin{array}{c} 3\\ 1\\ 2\end{array}, \begin{array}{c} 0\\ 1\\ 2\end{array},\dots$$

	\begin{lem}
		\label{Consistent}
		If a graph $G$ has a consistent $(p\!:\!q)$-colouring with clustering $c_1$, and 
		a graph $H$ has a $(q\!:\!r)$-colouring with clustering $c_2$, then
		$G\boxtimes H$ has a $(p\!:\!r)$-colouring with clustering $c_1c_2$. 
	\end{lem}
	
	\begin{proof}
		Let $\alpha$ be a consistent $(p\!:\!q)$-colouring of $G$ with
		clustering $c_1$, that is each vertex $x\in V(G)$ has colours
		$\alpha^1_x,\ldots,\alpha^q_x$ such that $\alpha_x^{i} \neq
		\alpha_y^{j}$ for each edge $xy\in E(G)$ and for all distinct
		$i,j\in[1,q]$. Let $\beta$ be a
		$(q\!:\!r)$-colouring of $H$ with clustering $c_2$, with colours from
		$[q]$. Colour each vertex  $(x,v)$ of $G\boxtimes H$ by $\{\alpha_x^i:i\in \beta(v)\} \in \binom{[p]}{r}$. 
		
		Let $Z$ be a monochromatic component of $G\boxtimes H$ defined by colour $c\in[p]$. Let $(x,v)$ be any vertex in $Z$. So $c=\alpha_x^i$ for some $i\in \beta(v)$. Let $A_x$ be the $\alpha$-component of $G$ defined by $c$ and containing $x$. Let $B_v$ be the $\beta$-component of $H$ defined by $i$ and containing $v$. 
		
		Consider an edge $(x,v)(y,w)$ of $Z$. Thus $c=\alpha_y^j$ for some $j\in \beta(w)$. Since $A_x$ is an $\alpha$-component containing $x$ and ($xy\in E(G)$ or $x=y$), the vertex $y$ is also in $A_x$. 
		If $xy\in E(G)$, then $i=j$ since $\alpha$ is consistent. Otherwise, $x=y$ and $\alpha_x^i=c=\alpha_y^j =\alpha_x^j$, again implying $i=j$. In both cases $i=j \in \beta(w)$. Since $B_v$ is a $\beta$-component defined by $i$ and containing $v$ and ($vw\in E(G)$ or $v=w$), the vertex $w$ is also in $B_v$. 
		
		For every edge $(x,v)(y,w)$ of $Z$, we have shown that $y\in V(A_x)$ and $w\in V(B_v)$. Thus $A_y=A_x$ and $B_w=B_v$. Since $Z$ is connected, for all vertices $(x,v)$ and $(y,w)$ of $Z$, we have $A_x=A_y$ and $B_v=B_w$. Hence $Z\subseteq A_x\boxtimes B_v$ for any $(x,v)\in V(Z)$. 
		
		Since $A_x$ is $\alpha$-monochromatic, $|A_x|\leq c_1$. 
		Since $B_v$ is $\beta$-monochromatic, $|B_v|\leq c_2$. As $Z\subseteq A_x\boxtimes B_v$, $|Z| \leq |A_x \boxtimes B_v| \leq c_1c_2$. 
		Hence, our colouring of $G\boxtimes H$ has clustering $c_1c_2$.
	\end{proof}
	
	Every proper colouring is consistent, so \cref{Consistent} implies:
	
	\begin{cor}
		\label{ConsistentClustered}
		If a graph $G$ has a proper $(p\!:\!q)$-colouring, and a graph $H$ has a $(q\!:\!r)$-colouring with clustering $c$, then  $G\boxtimes H$ has a $(p\!:\!r)$-colouring with clustering $c$.
	\end{cor}
	
	\begin{cor}
		\label{ConsistentProperProper}
		If a graph $G$ has a proper $(p\!:\!q)$-colouring and a graph $H$ has a proper $(q\!:\!r)$-colouring, then $G\boxtimes H$ has a proper $(p\!:\!r)$-colouring.
	\end{cor}

\cref{ClusteredProductColouring} states that a $(p_1\!:\!q_1)$-colouring
of a graph $G$ (with bounded clustering) can be combined with a $(p_2\!:\!q_2)$-colouring
of a graph $H$ (with bounded clustering) to produce a $(p_1p_2\!:\!q_1q_2)$-colouring
of $G\boxtimes H$ (with bounded clustering). A natural
question is whether this fractional colouring can be simplified; that
is, is there a $(p_3\! :\!q_3)$-colouring
of $G\boxtimes H$ with
$\tfrac{p_3}{q_3}\le\tfrac{p_1p_2}{q_1q_2}$ and $p_3<p_1p_2$? 
There is no hope to obtain such a simplification in general, since if $G$ and $H$ are complete graphs and $q_1=q_2=1$, then $G\boxtimes H$ is a complete graph on $p_1p_2$ vertices. However,
\cref{ConsistentClustered} shows that when the fractional colouring of $G$ is proper
and $q_1=p_2$, the resulting fractional colouring of $G\boxtimes H$ can
be simplified significantly.

We now show another way to simplify the $(p_1p_2\!:\!q_1q_2)$-colouring
of the graph  $G\boxtimes H$, by allowing a
small loss on the fraction $\tfrac{p_1p_2}{q_1q_2}$. Below we  only
consider the case $p_1=p_2$ and $q_1=q_2$ for simplicity, but the
technique can be extended to the more general case. We use the Chernoff bound: For any $0 \le t \le nx$, the probability that the binomial distribution
$\mathrm{Bin}(n,x) $ with parameters $n$ and $x$ differs from its
expectation $nx$ by at least $t$ satisfies
  \[\Pr(|\mathrm{Bin}(n,x)-nx| > t) < 2\exp(-t^2/(3nx)).\]

\begin{lem}
	Assume that $G$ has a $(p\!:\!q)$-colouring (with bounded clustering) and $H$ has
	a $(p\!:\!q)$-colouring (with bounded clustering). Then for any real
	number $0<x\le  1$,
	$G\boxtimes H$ has a $\Big(xp^2+O(p\sqrt{x}):\big(xq^2-O(q^{3/2}\sqrt{x\log p})\big)\Big)$-colouring (with bounded clustering).
\end{lem}

\begin{proof}
	Let $X$ be a random subset of $[p]^2$ obtained by including each
	element of $[p]^2$ independently with probability $x$. By the
	Chernoff bound, $p':=|X|\le xp^2+O(p\sqrt{x})$ with high probability
	(i.e., with probability tending to $1
	$ as $p\to \infty$).
	
	Consider two $q$-element subsets  $S,T\subseteq [p]$. Then it
	follows from the Chernoff bound that for any $0\le t \le xq^2$, the probability that $S\times
	T$ contains less than $xq^2-t$ elements of $X$ is at most
	$2\exp(-t^2/(3q^2x))$. By the union bound, the probability that
	there exist two $q$-element subsets  $S,T\subseteq [p]$ with
	$|(S\times T)\cap X|<xq^2-t$ is at most $p^q\cdot p^q\cdot
	2\exp(-t^2/(3q^2x))<2\exp(2q\log p-t^2/(3q^2x))$. By taking
	$t=\Theta(q^{3/2}\sqrt{x\log p})$, this quantity is less than
	$\tfrac12$. 
	
	It follows that there exists a subset $X\subseteq [p]^2$ of at most
	$p'=xp^2+O(p\sqrt{x})$ elements, such that for all $q$-elements subsets
	$S,T\subseteq [p]$, 
	$|(S\times T)\cap X|\ge q':=xq^2-O(q^{3/2}\sqrt{x\log p})$.
	
	\smallskip
	
	Let $c_G$ be a $(p\!:\!q)$-colouring of $G$ with colours $[p]$ and
	let $c_H$ be a $(p\!:\!q)$-colouring of $H$ with colours $[p]$.
	For any pair $(i,j)\in X$ define the colour class 
	$C_{ij}=\{u\in V(G)\,:\, i \in c_G(u)\}\times
	\{v\in V(H)\,:\, j\in c_H(v)\} $ in $G\boxtimes H$.

	Let $c$ denote the resulting colouring of
	$G\boxtimes H$, and observe that if $c_G$ and $c_H$ are proper, so
	is $c$, and if  $c_G$ and $c_H$ have bounded clustering, so does
	$c$, since each colour class in $c$ is the cartesian product of a
	colour class of $G$ and a colour class of $H$. Moreover $c$ uses at
	most $p'$ colours, and each vertex of $G\boxtimes H$ receives at least $q'$
	colours. It follows that $c$ is a $(p'\!:\!q')$-colouring of
	$G\boxtimes H$ (with bounded clustering), as desired.
\end{proof}

\subsection{Paths and Cycles}\label{sec:pathscycles}

The next lemma shows how to obtain a consistent colouring of a tree with small clustering. 

\begin{lem}
	\label{EdgePartitionConsistentColouring}
	If a tree $T$ has an edge-partition $E_1,\dots,E_k$ such that for each $i\in[1,k]$  
	each component of $T-E_i$ has at most $c$ vertices, then $T$ has a consistent $(k+1\!:\!k)$-colouring with clustering $c$. 
\end{lem}

\begin{proof} 
	Root $T$ at some leaf vertex $r$ and orient $T$ away from
	$r$. We now label each vertex $v$ of $T$ by a sequence
	$(\ell^1_v,\dots,\ell^k_v)$ of distinct elements of
	$\{1,\dots,k+1\}$. First label $r$ by $(1,\dots,k)$. Now label
	vertices in order of non-decreasing distance from
	$r$. Consider an arc $vw$ with $v$ labelled
	$(\ell^1_v,\dots,\ell^k_v)$ and $w$ unlabelled. Say $vw\in
	E_i$. Let $z$ be the element of
	$\{1,\dots,k+1\}\setminus\{\ell^1_v,\dots,\ell^k_v\}$. Then
	label $w$ by
	$(\ell^1_v,\dots,\ell^{i-1}_v,z,\ell^{i+1}_v,\dots,\ell^k_v)$. Label
	every vertex in $T$ by repeated application of this rule. It
	is immediate that this labelling determines a consistent
	$(k+1\!:\!k)$-colouring of $T$. 	
	
	Consider a monochromatic component $X$ of $T$ determined by colour $i$. If $vw$ is an edge of $T$ with $v\in V(X)$ and $w\not\in V(X)$ and $vw\in E_j$, then $\ell_j(v)=i$ and $\ell_j(w)\neq i$ and the only colour not assigned to $w$ is $\ell_j(v)=i$. By consistency, $\ell_j(x)=i$ for every $x \in V(X)$, and for every edge $xy\in E(T)$ with $x\in V(X)$ and $y\not\in V(X)$ we have $xy\in E_j$. Thus $X$ is contained in $T-E_j$ and has at most $c$ vertices. 
\end{proof}

\begin{lem}
	\label{ConsistentPath}
	For every $k\in\NN$, every path has a consistent $(k+1\!:\!k)$-colouring with clustering $k$. 
\end{lem}

\begin{proof}
	Let $e_1,\dots,e_n$ be the sequence of edges in a path $P$. For $i\in[0,k-1]$ let $E_i:=\{e_j: j \equiv i \pmod{k}\}$. So $E_0,\dots,E_{k-1}$ is an edge-partition of $P$ such that for each $i\in[0,k-1]$  each component of $T-E_i$ has at most $k$ vertices. By \cref{EdgePartitionConsistentColouring}, $P$ has a consistent $(k+1\!:\!k)$ colouring with clustering $k$. 
\end{proof}

\cref{ConsistentPath,Consistent} imply the following result. Products with paths are of particular interest because of the results in \cref{SubgraphsStrongProducts}.

\begin{thm}
\label{MultiplyPath}
If a graph $G$ is $k$-colourable with clustering $c$ and $P$ is a path, then $G\boxtimes P$ is $(k+1)$-colourable with clustering $ck$.
\end{thm}

Recall that $\boxtimes_d P_{n}$ denotes the $d$-dimensional grid $P_n\boxtimes \cdots \boxtimes P_n$. \cref{MultiplyPath} implies the upper bound in the following result. As discussed in \cref{sec:hex}, the lower bound comes from the $d$-dimensional Hex Lemma~\citep{Gale79,LMST08,MP08,BDN17,MP08,Matdinov13,Karasev13}. 

\begin{cor}
	\label{HexGrid}
	$\boxtimes_d P_{n}$ is $(d+1)$-colourable with clustering $d!$. Conversely, every $d$-colouring of $\boxtimes_d P_{n}$ has a monochromatic component of size at least $n$. Hence $$\cchi(\{ \boxtimes_d P_{n}  : n \in\NN\})=d+1.$$
\end{cor}

Note that the corollary above can also be deduced from the
following simple lemma, which does not use consistent
colourings (however we need this notion to prove the stronger
\cref{MultiplyPath} above, and its generalisation \cref{MultiplyTree} in \cref{sec:treestreewidth}). 

\begin{lem}\label{lem:comment}
	If $G$ is $(p\!:\!q)$-colourable with clustering $c_1$ and $H$ is
	$(p\!:\!r)$-colourable with clustering $c_2$, and $q+r>p$, then $G\boxtimes H$ is
	$(p\!:\!(q+r-p))$-colourable with clustering $c_1c_2$.
\end{lem}

\begin{proof}
	Consider a $(p\!:\!q)$-colouring of $G$ and a
	$(p\!:\!r)$-colouring of $H$ and for each $i\in [p]$, let the colour class
	of colour $i$ in $G\boxtimes H$ be the product of the colour class of
	colour $i$ in $G$ and the colour class of
	colour $i$ in $H$. Clearly, monochromatic components in $G\boxtimes H$
	have size at most $c_1c_2$. Moreover, a pigeonhole argument tells us
	that each vertex of $G\boxtimes H$ is covered by at least  $q+r-p$
	colours in $G\boxtimes H$, as desired.
\end{proof}

In particular \cref{lem:comment}  shows that if $G$ is $(p\!:\!q)$-colourable with
clustering $c$, and $P$ is a path,  then $G\boxtimes P$ is $(p\!:\!q-1)$-colourable
with clustering $(p-1)c$. Here we have used the statement of
\cref{ConsistentPath}, that every path has a $(p\!:\!p-1)$-colouring
with clustering $p-1$ (but we did not use the additional property that such a
colouring could be taken to be consistent). By induction this easily implies \cref{HexGrid}.

\medskip

It is an interesting open problem to determine the minimum clustering
function in a $(d+1)$-colouring of $\boxtimes_d P_{n}$. Since $\boxtimes_d P_{n}$ contains a $2^d$-clique, every $(d+1)$-colouring has a monochromatic component with at least $2^d/(d+1)$ vertices. 

The fractional clustered chromatic number of Hex grid graphs is very different from the clustered chromatic number. 

\begin{prop}
	\label{cfchi-hexgrid}
	For fixed $d\in\NN$, 
	$$\cfchi(\{ \boxtimes_d P_{n}  : n \in\NN\})=1.$$
\end{prop}

\begin{proof}
	Fix $\epsilon\in(0,1)$ and let $k:= \ceil{2d/\epsilon}$. By \cref{ConsistentPath}, every path has a  $(k+1\!:\!k)$-colouring with clustering $k$. By \cref{ClusteredProductColouring}, for every $n\in\NN$, the graph $\boxtimes_d P_n$ is $((k+1)^d\!:\!k^d)$-colourable with clustering $k^d$. For $k\geq 2d$, it is easily proved by induction on $d$ that $(k+1)^d/k^d \leq 1 + 2d/k$. Thus 
	$(k+1)^d/k^d \leq 1+\epsilon$. This says that for any $\epsilon>0$ there exists $c$ (namely, $\ceil{2d/\epsilon}^d$) such that for every $n\in\NN$, the graph $\boxtimes_d P_n$ is fractionally $(1+\epsilon)$-colourable with clustering $c$. The result follows. 
	%
	%
	%
	%
	%
	%
	%
	%
	%
	%
	%
	%
\end{proof}


\cref{cfchi-hexgrid} can also be deduced from a result of
\citet{Dvorak16} (who proved that the conclusion holds for any class
of bounded degree having sublinear separators). It can also be deduced
from a result of \citet{BDLM08}, which states that classes of bounded
asymptotic dimension have fractional asymptotic dimension 1 (combined
with the discussion of \cref{sec:asdim} on the connection between
asymptotic dimension and clustered colouring for classes of graphs of
bounded degree). Note that the main result of \cite{BDLM08}, which
states that if $\mathcal{F}_1$ and $\mathcal{F}_2$ have asymptotic
dimension $m_1$ and $m_2$, respectively, then $\mathcal{F}_1\boxtimes
\mathcal{F}_2$ has asymptotic dimension $m_1+m_2$, is obtained by
combining the result on fractional asymptotic dimension mentioned
above with an elaborate version of \cref{lem:comment}.

\begin{lem}
	\label{ConsistentCycle}
	For every $k\in\NN$, every cycle has a consistent $(k+1\!:\!k)$-colouring with clustering $k^2+3k-1$. 
\end{lem}

\begin{proof}
	Let $C=(v_1,\dots,v_n)$ be an $n$-vertex cycle. Consider integers $a$ and $b\in[0,k(k+1)-1]$ such that $n= ak(k+1)+b$. By \cref{ConsistentPath}, the path $(v_1,\dots,v_{n-b})$ has a consistent $(k+1\!:\!k)$-colouring with clustering $k$. Observe that the colour sequences assigned to vertices repeat every $k(k+1)$ vertices. Thus $v_1$ and $v_{n-b}$ are assigned the same sequence of $k$ colours. Give this colour sequence to all of $v_{n-b+1},\dots,v_n$. We obtain a consistent $(k+1\!:\!k)$-colouring of $C$ with clustering $k+k(k+1)-1+k=k^2+3k-1$.
\end{proof}

\cref{ConsistentPath,ConsistentCycle} imply that for every
$\epsilon>0$ there exists $c\in\NN$, such that every graph with
maximum degree 2 is fractionally $(1+\epsilon)$-colourable with clustering $c$. 
Thus the fractional clustered chromatic number of graphs with maximum
degree 2 equals 1. The following open problem naturally arises:

\begin{qu}\label{q:clusteringDelta}
	What is the fractional clustered chromatic number of graphs with
	maximum degree $\Delta$?
\end{qu}

We now show that the same lower bound from the non-fractional setting (see \citep{WoodSurvey}) holds in the fractional setting.

\begin{prop}
	For every even integer $\Delta$, the fractional clustered chromatic number of the class of graphs with maximum degree $\Delta$ is at least $\frac{\Delta}{4}+\half$.
\end{prop}

\begin{proof}
	
	
%
%
        We need to prove that for any even integer $\Delta$ and any integer $c$, there is a graph
        $G$ of maximum degree $\Delta$ such that for any integers $p,q$, if 
        $G$ is $(p\!:\!q)$-colourable with clustering $c$, then $\tfrac{p}{q}\ge\frac{\Delta}{4}+\half$.

        Fix an even integer $\Delta$ and an integer $c$, and consider a
        $(\frac{\Delta}{2}+1)$-regular graph $H$ with girth greater than
	$c$ (such a graph exists, as proved by \citet{ES63}). Let $G=L(H)$
	be the line-graph of $H$. Note that $G$ is $\Delta$-regular. Let $p,q$
	 be such that
	$G$ is $(p\!:\!q)$-colourable with clustering $c$, and let $f$ be such
	a colouring. 
	Consider an $f$-monochromatic subgraph of $G$ with vertex set $X$, so every component
	of $G[X]$ has at most $c$ vertices. Let $F_X$ be the set of edges of $H$
	corresponding to the vertices of $X$ in $G$.
	
	Since $H$ has girth greater than $c$, the subgraph of $H$ determined
	by the edges of $F$ is a forest, and thus $|X|=|F_X|< |V(H)|$. There are $p$ such monochromatic subgraphs and each vertex of $G$ is in exactly $q$ such subgraphs. Thus 
	$$ q \half (\tfrac{\Delta}{2}+1) |V(H)| = q |E(H)| = q |V(G)| = \sum_X |X| = \sum_{X} |F_X| < p |V(H)|.$$
	Hence
	$\frac{p}{q} >
	\frac{\Delta}{4} + \half $, as desired.
\end{proof}

The $\Delta=3$ case of \cref{q:clusteringDelta} is an interesting problem. The line graph of the 1-subdivision of a high girth cubic graph provides a lower bound of $\frac{6}{5}$ on the fractional clustered chromatic number.

\subsection{Trees and Treewidth}\label{sec:treestreewidth}

\cref{ConsistentPath} is generalised for bounded degree trees as follows:

\begin{lem}
	\label{ConsistentTree}
	For all $k,\Delta\in\NN$, every tree $T$ with maximum degree $\Delta\geq 3$ has a consistent $(k+1\!:\!k)$-colouring with clustering less than  $2(\Delta-1)^{k-1}$.
\end{lem}

\begin{proof}
	If $k=1$ then a proper 2-colouring of $T$ suffices. Now assume that $k\geq 2$. Root $T$ at some leaf vertex $r$. Consider the edge-partition $E_0,\dots,E_{k-1}$ of $T$, where $E_i$ is the set of edges $uv$ in $T$ such that $u$ is the parent of $v$ and $\dist_T(r,u)\equiv i\pmod{k}$. Each component $X$ of $T-E_i$ has height at most $k-1$ and each vertex $v$ in $X$ has at most $\Delta-1$ children in $X$, implying 
	$|V(X)|\leq 	1+(\Delta-1)+(\Delta-1)^2+\dots+(\Delta-1)^{k-1}=
	\frac{(\Delta-1)^k-1}{\Delta-2}
	<2(\Delta-1)^{k-1}$. 	The result then follows from \cref{EdgePartitionConsistentColouring}. 
\end{proof}

\cref{Consistent,ConsistentTree} imply the following generalisation of 
\cref{MultiplyPath}:

\begin{lem}
\label{MultiplyTree}
If a graph $G$ is $k$-colourable with clustering $c$ and $T$ is a tree with maximum degree $\Delta\geq 3$, then $G\boxtimes T$ is $(k+1)$-colourable with clustering less than $2c(\Delta-1)^{k-1}$. 
\end{lem}

\cref{MultiplyTree} leads to the next result. We emphasise that $T_1$
may have arbitrarily large maximum degree (if $T_1$ also has bounded
degree then the result is again a simple consequence of
\cref{lem:comment}, which does not use consistent colourings). 

\begin{thm}
	\label{TreeProduct}
	If $T_1,\dots,T_d$ are trees, such that each of $T_2,\dots,T_d$ have maximum degree at most $\Delta\geq3$, then $T_1\boxtimes \dots\boxtimes T_d$ is $(d+1)$-colourable with clustering less than $2^d(\Delta-1)^{ \binom{d}{2} }$.
\end{thm}

\begin{proof}
	We proceed by induction on $d\geq 1$. In the base case, $T_1$ is 2-colourable with clustering 1. Now assume that $T_1\boxtimes \dots \boxtimes T_{d-1}$ is $d$-colourable with clustering less than $2^{d-1}(\Delta-1)^{ \binom{d-1}{2} }$. \cref{MultiplyTree} with $G=T_1\boxtimes\dots\boxtimes T_{d-1}$ implies that $T_1\boxtimes\dots\boxtimes T_{d}$ is $(d+1)$-colourable with clustering  less than $2\cdot 2^{d-1}(\Delta-1)^{ \binom{d-1}{2} } (\Delta-1)^{d-1} = 2^d(\Delta-1)^{ \binom{d}{2} }$. 
\end{proof}


\cref{TreeProduct} is in sharp contrast with \cref{StarProduct}: for the strong product of $d$ stars we need $2^d$ colours even for defective colouring, whereas for bounded degree trees we only need $d+1$ colours in the stronger setting of clustered colouring. This highlights the power of assuming bounded degree in the above results. 

Let	$\TT_k$ be the class of graphs with treewidth $k$. Such graphs
are $k$-degenerate and $(k+1)$-colourable. Since the graph $C_{n,k}$ (defined in
\cref{sec:ptc}) has treewidth $k-1$, \cref{FracClosureProduct} implies that 
$$\dfchi(\TT_k) = \cfchi(\TT_k) = \fchi(\TT_k) = \dchi(\TT_k) = \cchi(\TT_k) = \bigchi(\TT_k) = k+1.$$
\citet{ADOV03} showed that graphs of bounded treewidth and bounded
degree are 2-colourable with bounded clustering. Note that
\cref{DegreeTreewidthClustered} in \cref{sec:asdim} generalises this result ($d=1$) and generalises \cref{TreeProduct}
($k=1$). We now give a short proof of \cref{DegreeTreewidthClustered}
(which we restate below for convenience) that does not use asymptotic dimension, or any results related to it.

\DegreeTreewidthClustered*

\begin{proof}
	By \cref{DegreeTreewidthStructure}, $G_i$ is a subgraph of $T_i \boxtimes K_{20k\Delta}$ for some tree $T_i$ with maximum degree at most $20k\Delta^2$. Thus  $G_1\boxtimes \dots\boxtimes G_d$ is  a subgraph of $ ( T_1 \boxtimes K_{20k\Delta} ) \boxtimes \dots \boxtimes ( T_d \boxtimes K_{20k\Delta} )$, which is a subgraph of $(T_1\boxtimes \dots\boxtimes T_d ) \boxtimes  K_t$, where $t:= (20k\Delta)^d$.  By \cref{TreeProduct},  $T_1\boxtimes \dots\boxtimes T_d$ is $(d+1)$-colourable with clustering at most some function $c'(d,k,\Delta)$. By \cref{BlowUp}, $(T_1\boxtimes \dots\boxtimes T_d ) \boxtimes  K_t$ and thus $G_1\boxtimes \dots\boxtimes G_d$ is $(d+1)$-colourable with clustering $c(d,k,\Delta) :=c'(d,k,\Delta) \cdot t$.
\end{proof}

The next result shows that for any sequence of non-trivial classes
$\GG_1,\dots,\GG_d$, the bound on the number of colours in
\cref{DegreeTreewidthClustered} is best possible.

\begin{thm}
	If $\GG_1,\dots,\GG_d$ are graph classes with $\cchi(\GG_i)\geq 2$ for each $i\in\{1,\dots,d\}$, then $\cchi(\GG_1\boxtimes \dots \boxtimes \GG_d) \geq d+1$.
\end{thm}

\begin{proof}
	By replacing each class $\GG_i$ by its monotone closure if
	necessary, we can assume without loss of generality that each class
	$\GG_i$ is monotone (i.e., closed under taking subgraphs).
	If there is a constant $d$ such that every component of a graph of $\GG_i$ has
	maximum degree at most $d$ and diameter at most $d$, then 
	$\cchi(\GG_i)\le 1$. It follows that for any $1\le i\le d$, the
	graphs of the class  $\GG_i$ contain arbitrarily large degree
	vertices or arbitrarily long paths. By monotonicity, it follows that
	for some constant $1\le k \le d$,
	$\GG_1\boxtimes \dots \boxtimes \GG_d$ contains the class
	$(\boxtimes_k \mathcal{P} )\boxtimes (\boxtimes_{d-k}\mathcal{S})$,
	where $\mathcal{P}$ denotes the class of all paths, and
	$\mathcal{S}$ denotes the class of all stars.
	
	We claim that for any graph $G$, $\cchi(G\boxtimes \mathcal{P})\le
	\cchi(G\boxtimes \mathcal{S})$. To see this, assume that $G\boxtimes
	\mathcal{S}$ is $\ell$-colourable with clustering $c$, and consider
	an $\ell$-colouring $f$ of $G\boxtimes K_{1,n}$ with clustering $c$, with $n=c\cdot
	\ell^{|V(G)|}$. The graph $G\boxtimes K_{1,n}$ can be considered as
	the union of $n+1$ copies of $G$, one copy for the centre of the star
	$K_{1,n}$ (call it the central copy of $G$), and $n$ copies for the leafs
	of $K_{1,n}$ (call them the leaf copies of $G$). By the pigeonhole
	principle, at least $c$ leaf copies $G_1,\ldots,G_c$ of $G$ have precisely the same
	colouring, that is for each vertex $u$ of $G$, and any two copies
	$G_i$ and $G_j$ with $1\le i<j\le c$, the two copies of $u$ in $G_i$
	and $G_j$ have the same colour in $f$. Let us denote this colouring of
	$G$ by $f_l$, and let us denote the restriction of $f$ to the
	central copy by $f_c$ (considered as a colouring of $G$). Note that
	for any vertex $v$ of $G$ we have $f_l(v)\ne f_c(v)$, and for any
	edge $uv$ of $G$ we have $f_l(u)\ne f_c(v)$, otherwise $f$ would
	contain a monochromatic star on $c+1$ vertices. We can now obtain a
	colouring of $G\boxtimes P$, for any path $P$, by alternating the
	colourings $f_l$ and $f_c$ of $G$ along the path. This shows that
	$G\boxtimes \mathcal{P}$ is $\ell$-colourable with clustering $c$.
	
	It follows from the previous paragraph that $\cchi((\boxtimes_k
	\mathcal{P} )\boxtimes (\boxtimes_{d-k}\mathcal{S}))\ge
	\cchi(\boxtimes_k \mathcal{P} )$. By the Hex lemma (see
	\cref{sec:hex}), this implies $\cchi(\GG_1\boxtimes \dots \boxtimes
	\GG_d) \geq d+1$, as desired.
\end{proof}

\subsection{Graph Parameters}\label{sec:param}

We now explain how some results of this paper can be proved in a more general setting. 
For the sake of readability, we chose to present them (and prove them) 
only for the case of (clustered) colouring in the previous sections.

\medskip

A \defn{graph parameter} is a function $\eta$ such that $\eta(G)\in\RR\cup\{\infty\}$ for every graph $G$, and $\eta(G_1)=\eta(G_2)$ for all isomorphic graphs $G_1$ and $G_2$. For a graph parameter $\eta$ and a set of graphs $\GG$, let $\eta(\GG):=\sup\{\eta(G):G\in\GG\}$ (possibly $\infty$).

For a graph parameter $\eta$, a colouring $f:V(G)\to C$ of a graph $G$ has \defn{$\eta$-defect $d$} if  $\eta( G[ f^{-1}(i) ] ) \leq d$ for each $i\in C$. Then a graph class $\GG$ is \defn{$k$-colourable with bounded $\eta$} if there exists $d\in\RR$ such that every graph in $\GG$ has a $k$-colouring with $\eta$-defect $d$. 
Let $\bigchi_\eta(\GG)$ be the minimum integer $k$ such that $\GG$ is $k$-colourable with bounded $\eta$, called the \defn{$\eta$-bounded chromatic number}. 

Maximum degree, $\Delta$, is a graph parameter, and the $\Delta$-bounded chromatic number coincides with the defective chromatic number, both denoted $\dchi(\GG)$. 

Define $\star(G)$ to be the maximum number of vertices in a connected component of a graph $G$. Then $\star$ is a graph parameter, and the $\star$-bounded chromatic number coincides with the clustered chromatic number, both denoted $\cchi(\GG)$. 

These definitions  also capture the usual chromatic number. For every graph $G$, define 
\begin{equation*}
	\iota(G):=\begin{cases}
		1 & \text{ if $E(G)=\emptyset$}\\
		\infty & \text{ otherwise }\\
	\end{cases}
\end{equation*}
For $d\in\RR$, a colouring $f$ of $G$ has $\iota$-defect $d$ if and only if $f$ is proper. Then 
$\bigchi_\iota(\{G\})=\bigchi(G)$.

A graph parameter $\eta$ is \defn{$g$-well-behaved} with respect to a particular graph product $\ast\in\{\square,\boxtimes\}$ if:
\begin{enumerate}[(W1)]
	\item $\eta(H) \leq \eta(G)$ for every graph $G$ and every subgraph $H$ of $G$, 
	\item $\eta( G_1 \cup G_2) = \max\{ \eta(G_1), \eta(G_2) \}$ for all disjoint graphs $G_1$ and $G_2$. 
	\item $\eta( G_1 \ast G_2) \leq g( \eta(G_1), \eta(G_2))$ for all graphs $G_1$ and $G_2$. 
\end{enumerate} 
A graph parameter is \defn{well-behaved} if it is $g$-well-behaved for some function $g$. For example:
\begin{compactitem}
	\item $\Delta$ is $g$-well-behaved with respect to $\square$, where $g(\Delta_1,\Delta_2)=\Delta_1+\Delta_2$. 
	\item $\Delta$ is $g$-well-behaved with respect to $\boxtimes$, where $g(\Delta_1,\Delta_2)=(\Delta_1+1)(\Delta_2+1)-1$.
	\item $\star$ is $g$-well-behaved with respect to $\square$ or $\boxtimes$, where $g(\star_1,\star_2)=\star_1 \star_2$.
	\item $\iota$ is $g$-well-behaved with respect to $\square$ or $\boxtimes$, where $g(\iota_1,\iota_2)= \iota_1 \iota_2 $.
\end{compactitem}
On the other hand, some graph parameters are $g$-well-behaved for no function $g$.
One such example is treewidth, since if $G_1$ and $G_2$ are $n$-vertex paths, then $\tw(G_1)=\tw(G_2)=1$ but  $\tw(G_1\boxtimes G_2) \geq \tw(G_1 \CartProd G_2) = n$, implying there is no function $g$ for which  (W3) holds. 

Let $\eta$ be a graph parameter. A fractional colouring of a graph $G$ has \defn{$\eta$-defect $d$} if $\eta(X)\leq d$ for each monochromatic subgraph $X$ of $G$. 

A graph class $\GG$ is \defn{fractionally $t$-colourable with bounded
	$\eta$} if there exists $d\in\RR$ such that every graph in $\GG$ has
a fractional $t$-colouring with $\eta$-defect $d$. Let
$\bigchi^f_\eta(\GG)$ be the infimum of all  $t\in\RR^+$ such that
$\GG$ is fractionally $t$-colourable with bounded $\eta$, called the
\defn{fractional $\eta$-bounded chromatic number}.

The next lemma generalises \cref{ChromaticNumber,FractionalProduct}. 

\begin{lem}
	\label{ProductColouringEta}
	Let $\eta$ be a $g$-well-behaved parameter with respect to $\boxtimes$. Let $G_1,\dots,G_d$ be graphs, such that $G_i$ is $(p_i\!:\!q_i)$-colourable with $\eta$-defect $c_i$, for each $i\in[1,d]$. Then $G:= G_1\boxtimes \cdots\boxtimes G_d$ is $(\prod_i p_i\!:\!\prod_i q_i)$-colourable with $\eta$-defect $g( c_1, g( c_2, \ldots , g( c_{d-1}, c_d ) ) )$. 
\end{lem}

\begin{proof}
	For $i\in[1,d]$, let $\phi_i$ be a $(p_i\!:\!q_i)$-colouring of $G_i$ with $\eta$-defect $c_i$.
	Let $\phi$ be the colouring of $G$, where each vertex $v=(v_1,\dots,v_d)$ of $G$ is coloured $\phi(v) := \{ (a_1,\dots,a_d): a_i \in \phi_i(v_i), i \in [1,d] \}$. So each vertex of $G$ is assigned a set of $\prod_i q_i$ colours, and there are $\prod_i p_i$ colours in total. Let $X:= X_1\boxtimes\dots\boxtimes X_d$, where each $X_i$ is a monochromatic component of $G_i$ using colour $a_i$. Then $X$ is a monochromatic connected induced subgraph of $G$ using colour $(a_1,\dots,a_d)$. Consider any edge $(v_1,\dots,v_d)(w_1,\dots,w_d)$ of $G$ with $(v_1,\dots,v_d)\in V(X)$ and $(w_1,\dots,w_d)\not\in V(X)$. Thus $v_iw_i\in E(G_i)$ and  $w_i\not\in V(X_i)$ for some $i\in\{1,\dots,d\}$. Hence $a_i\not\in \phi_i(w_i)$, implying $(a_1,\dots,a_d)\not\in \phi(w)$. Hence $X$ is a monochromatic component of $G$ using colour $(a_1,\dots,a_d)$. It follows from (W3) by induction that 
	$|V(X)| \leq g( c_1, g( c_2, \ldots , g( c_{d-1}, c_d ) ) )$. 
	Hence $\phi$ has $\eta$-defect $g( c_1, g( c_2, \ldots , g( c_{d-1}, c_d ) ) )$. 
\end{proof}

\cref{ProductColouringEta} implies:

\begin{thm}
	\label{ProductColouringEtaClasses}
	Let $\eta$ be a well-behaved parameter with respect to $\boxtimes$. 
	For all graph classes $\GG_1,\dots,\GG_d$, 
	$$\bigchi_\eta( \GG_1 \boxtimes \dots\boxtimes \GG_d) \leq \prod_{i=1}^d \bigchi_\eta(\GG_i).$$
\end{thm}

We have the following special case of \cref{ProductColouringEta}. 

\begin{lem}
	\label{DefectiveProductColouring}
	For all graphs $G_1,\dots,G_d$, if each $G_i$ is $(p_i\!:\!q_i)$-colourable with defect $c_i$, then $G_1\boxtimes \cdots\boxtimes G_d$ is $(\prod_i p_i\!:\!\prod_i q_i)$-colourable with defect $\prod_i(1+c_i)-1$.
\end{lem}

\cref{DefectiveProductColouring} implies the following analogues of \cref{ChromaticNumber} for defective colouring.

\begin{thm}
	\label{DefectiveProductColouringClasses}
	For all graph classes $\GG_1,\dots,\GG_d$ 
	$$\dchi( \GG_1 \boxtimes \dots\boxtimes \GG_d) \leq \prod_{i=1}^d \dchi(\GG_i).$$
\end{thm}

This is a generalised version of \cref{Consistent}.

\begin{lem}
	\label{ConsistentGeneral}
	Let $\eta$ be a $g$-well-behaved graph parameter. 
	If a graph $G$ has a consistent $(p\!:\!q)$-colouring with $\eta$-defect $c_1$, and 
	a graph $H$ has a $(q\!:\!r)$-colouring with $\eta$-defect $c_2$. Then  $G\boxtimes H$ has a $(p\!:\!r)$-colouring with $\eta$-defect $g(c_1,c_2)$. 
\end{lem}

We also have the following generalised version of \cref{ConsistentClustered}.

\begin{lem}
	\label{ConsistentProperGeneral}
	Let $\eta$ be a $g$-well-behaved graph parameter. If a graph $G$ has a proper $(p\!:\!q)$-colouring, and a graph $H$ has a $(q\!:\!r)$-colouring with $\eta$-defect $c$, then  $G\boxtimes H$ has a $(p\!:\!r)$-colouring with $\eta$-defect $g(\eta(K_1),c_2)$. 
\end{lem}

\begin{cor}
	\label{ConsistentDefectiveGeneral}
	If a graph $G$ has a proper $(p\!:\!q)$-colouring, and a graph $H$ has a $(q\!:\!r)$-colouring with defect $c$, then  $G\boxtimes H$ has a $(p\!:\!r)$-colouring with defect $c$.
\end{cor}

\subsection*{Acknowledgements}
 
This research was initiated at the Graph Theory Workshop held at
Bellairs Research Institute in April 2019. We thank the other
workshop participants for creating a productive working
atmosphere (and in particular Vida Dujmovi{\'c} and  Bartosz Walczak for discussions related to the paper). Thanks to both referees for several insightful comments.



\def\soft#1{\leavevmode\setbox0=\hbox{h}\dimen7=\ht0\advance \dimen7
	by-1ex\relax\if t#1\relax\rlap{\raise.6\dimen7
		\hbox{\kern.3ex\char'47}}#1\relax\else\if T#1\relax
	\rlap{\raise.5\dimen7\hbox{\kern1.3ex\char'47}}#1\relax \else\if
	d#1\relax\rlap{\raise.5\dimen7\hbox{\kern.9ex \char'47}}#1\relax\else\if
	D#1\relax\rlap{\raise.5\dimen7 \hbox{\kern1.4ex\char'47}}#1\relax\else\if
	l#1\relax \rlap{\raise.5\dimen7\hbox{\kern.4ex\char'47}}#1\relax \else\if
	L#1\relax\rlap{\raise.5\dimen7\hbox{\kern.7ex
			\char'47}}#1\relax\else\message{accent \string\soft \space #1 not
		defined!}#1\relax\fi\fi\fi\fi\fi\fi}

\end{document}